\documentclass[11pt]{article}
\usepackage{max}

\title{Repeated Angles in the Plane for Angles with Algebraic Tangents}
\author{Max Aires \\ 
Department of Mathematics, Rutgers University, New Brunswick, NJ\\ 
mla162@math.rutgers.edu}

\begin{document}

\maketitle

\begin{abstract}
We construct a set of points with $\Omega(n^2\log n)$ triples determining an angle $\th$ whenever $\tan(\th)$ is algebraic over $\QQ$, matching the upper bound of Pach and Sharir. This improves upon the original construction, which was optimal only for $\tan(\th)=a\s{m}/b$ with $a,b,m$ positive integers.
\end{abstract}

\section{Introduction}

Pach and Sharir \cite{PaSh92} proved that, for all $0<\th<\pi$, any set of $n$ points in the Euclidean plane contains at most $O(n^2\log n)$ triples which determine an angle with measure $\th$. They also showed that, for $\tan(\th)=\frac{a\s m}{b}$ with $a,b,m$ positive integers, this upper bound is tight, and that there exists a configuration of $n$ points in the plane with $\Omega_\th(n^2\log n)$ triples determining angle $\th$. However, for other $\th$, it remains open whether the upper bound is asymptotically tight (see problem 6, Chapter 6.2 in \cite{BrMoPa}). 

We shall construct a configuration of $n$ points with $\Omega_\th(n^2\log n)$ triples determining angle $\th$ for any $0<\th<\pi$ where $\tan(\th)$ is algebraic over $\QQ$. This class contains many angles of geometric interest; in particular, this gives a construction with $\Omega(n^2\log n)$ triples if $\th$ is any rational multiple of $\pi$, as well as for any constructible angle $\th$.

The question remains open in the case that $\tan(\th)$ is transcendental. We note the similar problem of determining the maximum number of subsets similar to a desired set $S$; Laczkovich and Ruzsa~\cite{LaRu} showed that there exist configurations with $\Omega(n^2)$ copies, matching the trivial upper bound, if and only if the cross-ratio of every quadruple in $S$ is algebraic. Their construction builds upon that of Elekes and Erd\H{o}s~\cite{ElEr91}, which finds similar copies within ``pseudo-grids,'' which are grid-like configurations of points made from generalized arithemetic progressions. Our construction makes use of a similar idea, generalizing the original approach of Pach and Sharir~\cite{PaSh92} from rectangle lattice grids to similar grid-like sets.

\section{Proof}

\begin{thm}
If $\tan(\th)$ is algebraic over $\QQ$, then for $n\ge 3$, there exist arrangements of $n$ points in the Euclidean plane with $\Omega_\th(n^2\ln n)$ triples determining angle $\th$.
\end{thm}

\begin{proof}
Let $\tan(\th)$ be a algebraic. Then we can write $\tan(\th)=\f{\al}{b}$, where $b\in\ZZ$ and $\al$ is a positive real algebraic integer. Then we can write $\al$ as the root of a monic, irreducible polynomial of degree $d$ in $\ZZ[x]$. So $\ZZ[\al]$ is a $\ZZ$-module with basis $1,\al,\dots,\al^{d-1}$.

We shall represent the Euclidean plane as $\CC$. Note that $b+i\al$ will have argument $\th$, so we seek triples of the form $z,z+\la_1v,z+(b+i\al)\la_2v$, with $z,v\in \CC$ and $\la_1,\la_2\in \RR$. Fix some positive integer $t$ (to be determined later), and define $\K_t$ and $\G_t$ as
$$\K_t:=\ll\{\sum_{k=0}^{d-1}a_k\al^k : a_k\in \ZZ, |a_k|\le t\rr\}$$
$$\G_t:=\K_t+i\K_t=\{u_1+iu_2:u_1,u_2\in \K_t\}$$
 
Note that $\G_t$ is the cartesian product of two real generalized arithmetic progression, while the pseduo-grids of Elekes and Erd\H{o}s~\cite{ElEr91} are one complex generalized arithmetic progression. Clearly, $|\K_t|=(2t+1)^d$ and $|\G_t|=(2t+1)^{2d}$. 

First, suppose $\mu\in \K_s$ and $\la\in \K_t$; clearly, $\mu+\la\in \K_{s+t}$. We shall now consider $\mu\la$. Since $\al$ is an algebraic integer with minimal polynomial of degree $d$, we can write $\al^d=c_0+c_1\al +\dots+c_{d-1}\al^{d-1}$ for $c_i\in \ZZ$; let $C_1=1+\max(|c_0|,|c_1|,\dots,|c_{d-1}|)$. By writing $\la=a_0+a_1\al+\dots+a_{d-1}\al^{d-1}$ with $a_i\in \ZZ$, $|a_i|\le t$, we see that $\al \la=(a_0\al+\dots+a_{d-2}\al^{d-1})+a_{d-1}(c_0+\dots+c_{d-1}\al^{d-1})\in \K_{C_1t}$. Applying this observation repeatedly, we see that $\al^k\la\in \K_{C_1^k t}$. Hence, if we write $\mu =b_0+b_1\al+\dots+b_{d-1}\al^{d-1}$, we see that $\mu\la=b_0\la+b_1\al \la +\dots+b_{d-1}\al^{d-1}\la\in \K_{dC_1^{d-1}st}$. 

Now suppose $u=(u_1+iu_2)\in \G_s$ and $v=(v_1+iv_2)\in \G_t$. Then $u+v=(u_1+u_2)+i(v_1+v_2)\in \G_{s+t}$, and $uv=(u_1+iu_2)(v_1+iv_2)=(u_1v_1-u_2v_2)+i(u_1v_2+u_2v_1)\in \G_{2dC_1^{d-1}st}$. Letting $C_2=2dC_1^{d-1}$, we have that $uv\in \G_{C_2st}$.

Ungar \cite{Un82} proved that a set of $N$ points forms has at least $N-1$ pairs of points each forming a distinct direction. By applying this to the set $\G_t$, we see that $\G_t$ contains at least $(2t+1)^{2d}-1\ge (2t)^{2d}$ pairs giving distinct directions. Since $\G_t-\G_t\subseteq \G_{2t}$, this means that the nonzero elements of $\G_{2t}$ determine at least $(2t)^{2d}$ distinct arguments (directions) from $[0,\pi)$.

Now, for each $1\le k\le t$, pick $T_k$ as $(2k)^{2d}-(2(k-1))^{2d}\ge (2k)^{2d-1}$ elements from $\G_{2k}$ whose angles are all distinct from each other and those of $T_1,\dots, T_{k-1}$.

For each choice of $v\in \bigcup_{k=1}^t T_k$ and each $0<\la\in \K_{\lf t/k\rf}$, the element $\la v$ is distinct, since each $v\in T_k$ has a distinct angle and each $\la$ is a unique positive real number. Moreover, if $v\in T_k\subseteq \G_{2k}$ and $\la\in \K_{\lf t/k\rf}$, $\la v\in \G_{C_2(2k)\lf t/k\rf}\subseteq \G_{2C_2t}$. Therefore, for any choice of $1\le k\le t$ and any $z\in \G_t$, $v\in T_k\subseteq\G_{2k}$, and $\la_1,\la_2\in \K_{\lf t/k\rf}$ with $\la_1,\la_2>0$, the triple 
$$(z,z+\la_1 v,z+(b+i \al)\la_2v)$$ 
is a unique triple determining angle $\th$. Observe that $z+\la_1v\in \G_{t+2C_2t}$ and (since $b+i\al\in \G_{|b|}$) $z+(b+\al i)\la_2v\in \G_{t+C_2|b|(2C_2t)}=\G_{(1+2|b|C_2^2)t}$, so for all choices of $z,v,\la_1,\la_2$, we have $z,z+\la_1v,z+(b+i\al)v\in \G_{(1+2|b|C_2^2)t}$. Since the elements of $\K_{\lf t/k\rf}$ are equally distributed between positive and negative, with one element $0$, the number of choices for $\la_1$ and $\la_2$ is $\ll(\f{|\K_{\lf t/k\rf}|-1}{2}\rr)= \f{(2\lf \f{t}{k}\rf+1)^{2d}-1}{2}\ge\lf\f{t}{k}\rf^{2d}$. Hence the total number of triples chosen in this way is
\begin{align*}
|\G_t|\sum_{k=1}^t|T_k|\lf\f{t}{k}\rf^{2d}
&\ge (2t+1)^{2d}\sum_{k=1}^t(2k)^{d-1}\ll(\f{t}{2k}\rr)^{2d}\\
&\ge t^{2d}\sum_{k=1}^t k^{2d-1}\ll(\f{t}{k}\rr)^{2d}\\
&=t^{4d}\sum_{k=1}^t \f{1}{k}\\
&\ge t^{4d}\ln t
\end{align*}
Hence, there are at least $t^{4d}\ln t$ triples which determine angle $\th$ with elements chosen from $\G_{(1+2|b|C_2^2)t}$. This set contains $(2(1+2|b|C_2^2)t+1)^{2d}\le (4|b|C_2^2+3)^{2d}t^{2d}$, so if we let $C_3=(4|b|C_2^2+3)^{2d}$, we have a set of at most $C_3t^{2d}$ points whose elements form at least $t^{4d}\ln t$ triples.

Now, for an arbitrary $n$, pick $t$ so that $C_3t^{2d}<n\le C_3(t+1)^{2d}$. If $n\ge C_3$, then we have a set with $C_3t^{2d}\le n$ points. Moreover, if $t\ge1$, then $n\le C_3(2t)^{2d}$, so $t\ge \f{(n/C_3)^{1/2d}}{2}$, so the number of triples determining angle $\th$ is at least $t^{4d}\ln t\ge \f{(n/C_3)^2}{2^{4d}}\ln \ll(\f{(n/C_3)^{1/2d}}{2}\rr)=\f{n^2}{2dC_3^22^{4d}}\ln \ll(\f{n}{2^{2d}C_3}\rr)$. Since $d$ and $C_3$ depend only on our angle $\th$ and not on $n$, we have a construction of at most $n$ points with $\Omega_\th(n^2\log n)$ triples determining angle $\th$, as desired.
\end{proof}


\begin{thebibliography}{1}

\bibitem{BrMoPa}
P.~Brass, W.~Moser, and J.~Pach.
\newblock Research problems in discrete geometry.
\newblock 2005.

\bibitem{ElEr91}
G.~Elekes and P.~Erd\H{o}s.
\newblock Similar configurations and pseudo grids.
\newblock {\em Intuitive geometry (Colloquia mathematica Societatis J\'{a}nos Bolyai)}, pages 85--104, 1991.

\bibitem{LaRu}
M.~Laczkovich and I.~Z.~Ruzsa.
\newblock The number of homothetic subsets.
\newblock {\em The Mathematics of Paul Erd\H{o}s II}, pages 294--302, 1997.

\bibitem{PaSh92}
J.~Pach and M.~Sharir.
\newblock Repeated angles in the plane and related problems.
\newblock {\em Journal of Combinatorial Theory, Ser. A}, 59:12--22, 1992.

\bibitem{Un82}
P.~Ungar.
\newblock $2N$ noncollinear points determine at least $2N$ directions.
\newblock {\em Journal of Combinatorial Theory, Ser. A}, 33:343--347, 1982.

\end{thebibliography}

\end{document}